\documentclass{amsart}
\usepackage{fontenc}
\usepackage{url}
\usepackage{hyperref}
\newtheorem{construction}{Construction}
\newtheorem{theorem}{Theorem}
\newtheorem{lemma}[theorem]{Lemma}
\title{A simple proof that the edge density of Fon-der-Flaass $(3,4)$-graph is $\geq\frac{7}{16}(1-o(1))$}
\author{Veronica Phan*}
\thanks{*Ho Chi Minh City; email: \url{kyubivulpes@gmail.com}}
\begin{document}
\begin{abstract}
In 2018, Alexander A. Razborov proved that the edge density of Fon-der-Flaass $(3,4)$-graph is $\geq\frac{7}{16}(1-o(1))$, using flag algebras. In this paper, we give an elementary proof of this result.
\end{abstract}
\maketitle
\section{Introduction}
Given a $3$-graph $G=(V,E)$, the edge density of $G$ is $\frac{|E|}{\binom{|V|}{3}}$. In 1941, Turán \cite{Turan} conjectured that the edge density of any $3$-graph without independent sets on $4$ vertices (or $(3,4)$-graph) is $\geq\frac{4}{9}(1-o(1))$. In 1988, Fon-der-Flaass \cite{FDF} gave a construction of a family of $(3,4)$-graph: 
\begin{construction}
Let $\Gamma=(V,A)$ be a simple directed graph contains no induced subgraph isomorphic to a direct circled of length $4$. We construct $3$-graph $F(\Gamma)=(V,F)$ which $(x,y,z)\in F$ if one of the following conditions hold:

$C1$: The induced subgraph in $\Gamma$ by $\{x,y,z\}$ has at most $1$ vertices.

$C2$: $(x,y),(x,z)\in A$.
\end{construction}
Fon-der-Flaass proved that the edge density of any $(3,4)$-graph that created this way is $\geq\frac{3}{7}(1-o(1))$, and in 2018, Alexander A. Razborov improved this bound to $\frac{7}{16}(1-o(1))$, using flag algebra. We will give an elementary proof of this result, without using flag algebra.

For technical reason, we remove the induced subgraph condition of directed graph $\Gamma$ and consider the complement $3$-graph of $F(\Gamma)$ instead.
\begin{construction}
\label{direct}
Let $\Gamma=(V,A)$ be a simple directed graph. We construct $3$-graph $CF(\Gamma)=(V,F)$ which $(x,y,z)\in F$ if the induced subgraph in $\Gamma$ by $\{x,y,z\}$ has at least $2$ vertices and $\{(x,y),(x,z)\},\{(y,z),(y,x)\},\{(z,x),(z,y)\}\not\subset A$.
\end{construction}
\begin{theorem}
\label{goal}
The edge density of $CF(\Gamma)$ of any directed graph $\Gamma$ is $\leq\frac{9}{16}(1+o(1))$
\end{theorem}
For a directed graph $\Gamma=(V,A)$ and the corresponding $3$-graph $CF(\Gamma)=(V,F)$, define the Lagrangian of them by:
$$\mathcal{L}_{CF}(\Gamma)=\sum_{(x,y,z)\in F}xyz+\frac{1}{2}\sum_{(x,y)\in A}x^2y$$
\begin{theorem}
\label{pre}
For any directed graph $\Gamma=(V,A)$ and any way of assigning non-negative real value to the elements of $V$ such that $\sum_{x\in V}=1$, we have $\mathcal{L}_{CF}(\Gamma)\leq\frac{3}{32}$.
\end{theorem}
By assigning $x=\frac{1}{|V|},\forall x\in V$, we can see that Theorem \ref{goal} is followed from Theorem \ref{pre}.
\section{Reduction to undirected graph}
We now construct a similar but simpler family of $3$-graphs from an undirected graph.
\begin{construction}
\label{undirect}
Let $G=(V,B)$ be a graph. We construct $3$-graph $BF(G)=(V,E)$ which $(x,y,z)\in E$ if the induced subgraph in $\Gamma$ by $\{x,y,z\}$ has at least $2$ vertices.
\end{construction}
For a graph $G=(V,B)$ and the corresponding $3$-graph $BF(G)=(V,E)$, define the Lagrangian of them by:
$$\mathcal{L}_{BF}(G)=\sum_{(x,y,z)\in E}xyz+\frac{1}{2}\sum_{(x,y)\in B}(x^2y+xy^2)-\frac{1}{2}\left(\sum_{(x,y)\in B}xy\right)^2$$
For directed graph $\Gamma=(V,A)$, construct $G(\Gamma)=(V,B)$ from $\Gamma$ by forget about all direction.
\begin{lemma}
\label{step}
For any directed graph $\Gamma=(V,A)$ and any way of assigning non-negative real value to the elements of $V$ such that $\sum_{x\in V}=1$, we have $\mathcal{L}_{CF}(\Gamma)\leq\mathcal{L}_{BF}(G(\Gamma))$.
\end{lemma}
\begin{proof}
Note that from construction \ref{direct} and \ref{undirect}, $(x,y,z)\in E/F$ if and only if exactly one of the following assertions is true: $(x,y),(x,z)\in A$ or $(y,z),(y,x)\in A$ or $(z,x),(z,y)\in A$. We then have
\begin{equation*}
\begin{split}
\mathcal{L}_{BF}(G(\Gamma))-\mathcal{L}_{CF}(\Gamma) &
=\sum_{(x,y),(x,z)\in A}xyz+\frac{1}{2}\sum_{(x,y)\in A}xy^2-\frac{1}{2}\left(\sum_{(x,y)\in B}xy\right)^2 \\ &
=\frac{1}{2}\sum_{x\in V}x\left(2\sum_{(x,y),(x,z)\in A}yz+\sum_{(x,y)\in A}y^2\right)-\frac{1}{2}\left(\sum_{(x,y)\in B}xy\right)^2 \\ &
=\frac{1}{2}\sum_{x\in V}x\left(\sum_{(x,y)\in A}y\right)^2-\frac{1}{2}\left(\sum_{(x,y)\in B}xy\right)^2 \\ &
=\frac{1}{2}\left(\sum_{x\in V}x\right)\sum_{x\in V}x\left(\sum_{(x,y)\in A}y\right)^2-\frac{1}{2}\left(\sum_{(x,y)\in B}xy\right)^2 \\ &
\geq\frac{1}{2}\left(\sum_{x\in V}x\left(\sum_{(x,y)\in A}y\right)\right)^2-\frac{1}{2}\left(\sum_{(x,y)\in B}xy\right)^2 \\ &
=\frac{1}{2}\left(\sum_{(x,y)\in A}xy\right)^2-\frac{1}{2}\left(\sum_{(x,y)\in B}xy\right)^2=0
\end{split}
\end{equation*}
by Cauchy-Schwarz and the assumption $x\in V$ are nonnegative, $\sum_{x\in V}=1$.
\end{proof}
From Lemma \ref{step}, to prove Theorem \ref{pre}, we only need to prove the following theorem:
\begin{theorem}
\label{main}
For any graph $G=(V,B)$ and any way to assign a non-negative real value to the elements of $V$ such that $\sum_{x\in V}=1$, we have $\mathcal{L}_{BF}(G)\leq\frac{3}{32}$.
\end{theorem}
\section{Reduction to complete graph}
Assume $G=(V,B)$ is not complete, then there exists $a,b\in V$ such that $(a,b)\notin B$. Define $G_a,G_b$ as the graphs created from $G$ by deleting the vertices $b,a$, respectively. For an assignment of nonnegative real value to the elements of $V$ such that $\sum_{x\in V}=1$, we assign the value of $G_a$ as follows: $x'=a+b$ if $x=a$ and $x'=x$ if otherwise, we do the same for $G_b$. Denote $S_{a,b}=\sum_{x\in V,(a,b,x)\in E}x,S_a=\sum_{x\in V,(a,x)\in B}x,S_b=\sum_{x\in V,b,x)\in B}x$. By construction \ref{undirect}, we have $S_{a,b}\leq\min(S_a,S_b)$, by the condition on the value assignment we have $0\leq S_a,S_b\leq1$, thus $|S_a-S_b|\leq1$.
\begin{lemma}
\label{reduce}
For the above notation, either $\mathcal{L}_{BF}(G)\leq\mathcal{L}_{BF}(G_a)$ or $\mathcal{L}_{BF}(G)\leq\mathcal{L}_{BF}(G_b)$.
\end{lemma}
\begin{proof}
If $a=0$, then we have $\mathcal{L}_{BF}(G)=\mathcal{L}_{BF}(G_a)$, so the lemma is true. Similar for $b=0$. So we can assume $a,b>0$.

By some algebra manipulation, we have:
\begin{equation*}
\begin{split}
a\mathcal{L}_{BF}(G_a)+b\mathcal{L}_{BF}(G_b)-(a+b)\mathcal{L}_{BF}(G) &
=ab(a+b)\left(\frac{1}{2}(S_a+S_b-(S_a-S_b)^2)-S_{a,b}\right) \\ &
=ab(a+b)\left(\frac{1}{2}(S_a+S_b-|S_a-S_b|)-S_{a,b}\right) \\ &
=ab(a+b)\left(\min(S_a,S_b)-S_{a,b}\right)\geq0
\end{split}
\end{equation*}
by the above observation that $S_{a,b}\leq\min(S_a,S_b)$ and $|S_a-S_b|\leq1$.

Because $a,b>0$, from the above inequality we have the result that we want.
\end{proof}
From Lemma \ref{reduce}, we only need to consider the case $G$ is complete.
\section{Complete the proof}
Now we can assume $G$ is complete, assume $V=\{x_1,x_2,...,x_n\}$, assign non-negative real value to them such that $\sum x_i=1$ then the Lagrangian become:
\begin{equation*}
\begin{split}
\mathcal{L}_{BF}(G) &
=\sum x_ix_jx_k+\frac{1}{2}\sum (x_i^2x_j+x_ix_j^2)-\frac{1}{2}\left(\sum x_ix_j\right)^2 \\ &
=\frac{1}{6}\left(1-\sum x_i^3\right)-\frac{1}{8}\left(1-\sum x_i^2\right)^2
\end{split}
\end{equation*}
We need to prove the Lagrangian is $\leq\frac{3}{32}$. We can assume $n\geq3$, by adding some variables $x_i=0$. Without lost of generality, assume $x_1\geq x_2\geq...\geq x_n$. Then $x_1+x_2+x_3\leq1$ and $\sum x_i^2\leq x_1^2+x_2^2+x_3(1-x_1-x_2)$. So we have:
$$\mathcal{L}_{BF}(G)\leq\frac{1}{6}\left(1-x_1^3-x_2^3-x_3^3\right)-\frac{1}{8}\left(1-x_1^2-x_2^2-x_3(1-x_1-x_2)\right)^2$$
which can be vertified to be $\leq\frac{3}{32}$ by Computer Algebra System like Mathematica.\cite{MSE}

\end{document}